\documentclass[10pt]{amsart}
\usepackage{caption,url}
\usepackage{subcaption}
\usepackage{amsmath}
\usepackage{amsthm, amssymb}

\makeatletter
    
    \@addtoreset{equation}{section}
  \makeatother

\newtheorem{thm}{Theorem}[section]
\newtheorem{coro}[thm]{Corollary}
\newtheorem{prop}[thm]{Proposition}
\newtheorem{lem}[thm]{Lemma}
\theoremstyle{definition}
\newtheorem{defi}[thm]{Definition}
\newtheorem{rmk}[thm]{Remark}
\newtheorem{cond}[thm]{Condition}

\newtheorem*{notation*}{Notation}

\newcommand{\G}{G}
\newcommand{\mcg}{\mathrm{MCG}(S)}

\newcommand{\pmf}{\mathcal{PMF}(S)}
\newcommand{\mf}{\mathcal{MF}(S)}
\newcommand{\ext}{\mathrm{Ext}}
\newcommand{\dt}{d_{\mathcal{T}}}
\newcommand{\Lip}{\mathrm{Lip}^{1}_{b}}
\newcommand{\hb}{\partial_{h}}
\newcommand{\PRs}{P\mathbb{R}_{\geq 0}^{\mathcal{S}}}

\title{On continuity of drifts of the mapping class group}
\author{Hidetoshi Masai}
\address{Department of Mathematics Tokyo Institute of Technology 2-12-1, Ookayama, Meguro-ku, Tokyo. 152-8551. Japan}
\email{masai@math.titech.ac.jp}
\begin{document}
\maketitle
\begin{abstract}
A random walk on a countable group $G$ acting on a metric space $X$ gives a characteristic called the drift which depends only on the transition probability measure $\mu$ of the random walk.
The drift is  the ``translation distance'' of the random walk.
In this paper, we prove that the drift varies continuously with the transition probability measures, 
under the assumption that the distance and the horofunctions on $X$ are expressed by certain ratios.
As an example, we consider
the mapping class group $\mcg$ acting on the Teichm\"uller space.
By using north-south dynamics, we also consider the continuity of the drift for a sequence converging to a Dirac measure.
As an appendix, we prove that the asymptotic entropy of the random walks on $\mcg$ varies continuously.
\end{abstract}

\section{Introduction}
For a group $G$ acting isometrically on a metric space $(X,d)$, 
the {\em drift} is one of the fundamental characteristics of random walks on $G$.
A random walk on $G$ can be specified by a probability measure $\mu$ on $G$, which will be the transition probability, and 
 $\mu$ induces a probability measure $\mathbb{P}$ on the space of sample paths $G^{\mathbb{Z}_{+}}$.
Let $\mu$ have finite first moment i.e. $\sum_{g\in G}\mu(g)d(gb,b)<\infty$,
where $b\in X$  is a base point.
Then the drift is defined for $\mathbb{P}$-a.e. sample path $\omega=(\omega_{n})_{n\in\mathbb{Z}_{+}}$, by
$$\ell(\mu):= \lim_{n\rightarrow\infty}\frac{d(\omega_{n}b,b)}{n}.$$
Here the notation $\ell(\mu)$ is reasonable as the drift depends only on the measure $\mu$, 
and is independent of the choice of the base point and sample paths almost surely.
Recall that a sequence $(\mu_{i})_{i\in\mathbb{N}}$ of probability measures on $G$ is said to {\it converge simply} to $\mu_{\infty}$ if 
$\mu_{i}(g)\rightarrow \mu_{\infty}(g)$ for every $g\in G$.
Given a sequence of measures $(\mu_{i})_{i\in\mathbb{Z}_{+}}$ converging simply to $\mu_{\infty}$,
it is natural to ask if $\ell(\mu_{i})$ converges to $\ell(\mu_{\infty})$.
When $G$ is Gromov hyperbolic acting on its Cayley graph, the continuity of the drift with respect to simple convergence of the probability measures is proved in \cite{EK,GMM}.

When we study random walks, it is often convenient (see e.g. \cite{MT}) to consider the {\em horofunctions} on $X$, and the {\em horoboundary} of $X$, see Section \ref{sec.horo} for definitions.
In this paper, we assume that the distance $d$ and the horofunctions are expressed in terms of certain ratios (see Section \ref{sec.conti} for a detail).
For example, Thurston's Lipschitz distance on the Teichm\"uller space $\mathcal{T}(S)$ 
between two points $x,y\in\mathcal{T}(S)$  is characterized in \cite{Thur} as 
$$d_{\mathrm{Th}}(x,y) = \log\sup_{\alpha\in\mathcal{S}}\frac{\mathrm{Len}_{y}(\alpha)}{\mathrm{Len}_{x}(\alpha)},$$
where $\mathcal{S}$ is the set of (isotopy classes of) simple closed curves and 
$\mathrm{Len}_{x}(\cdot)$ denotes the hyperbolic length given by $x\in\mathcal{T}(S)$.
In \cite{Wal}, Walsh showed that the horofunctions are also characterized similarly as a ratio, and thus the Teichm\"uller space equipped with Thurston's Lipschitz distance is an example which satisfies our assumption (see Section \ref{sec.Thur}).
Other examples discussed in this paper is the Teichm\"uller distance on the Teichm\"uller space.
The drift with respect to Thurston's Lipschitz distance is known to coincide with topological entropy and 
many other dynamical quantities, see e.g. \cite{Mas}.

Under the assumption that the distance and the horofunctions are written by certain ratios,
we prove that the drift $\ell(\mu)$ varies continuously with the probability measure $\mu$.
The key ingredient is the formula of the drift in terms of the integral of the Busemann cocycle (Lemma \ref{lem.drift}), which has already been observed in \cite{GMM} for the case of hyperbolic groups.
Lemma \ref{lem.drift} is a consequence of the comparison of the distance and the horofunction for a given sequence which converges to a ``boundary'' point (Lemma \ref{lem.compare}).
The versions of Lemma \ref{lem.compare} and Lemma \ref{lem.drift} 
for the Teichm\"uller distance 
is already given by Horbez \cite[Corollary 3.15.]{Hor}.
The proof by Horbez utilizes several deep works.
For example, the version of the Teichm\"uller distance is proved by using the recent work of Dowdall-Duchin-Masur \cite{DDM},
which gives a statistical hyperbolicity of the Teichm\"uller distance.
The purpose of this paper is to give a quick and unified proof of Lemma \ref{lem.compare} without using \cite{DDM}.
If $X$ is Gromov hyperbolic, Lemma \ref{lem.compare} follows from the standard arguments using $\delta$-thin triangles.
Therefore our discussion gives one quick way to observe ``statistical hyperbolicity'' of the Teichm\"uller space with respect to the distances mentioned above.
Furthermore, as we give a unified proof, we are able to observe an interesting statistical relationship between 
the hyperbolic length and the extremal length (Corollary \ref{cor. hyp and ext}).
In the above discussion, we suppose that the transition probabilities are non-elementary.
In section \ref{sec.NS}, we discuss the case where the transition probabilities converge to a Dirac measure 
whose mass is on an element $g\in G$ with north-south dynamics.
In this case, we prove that the drift converges to the translation distance of $g$.
This shows that the drift can be arbitrarily large,  any real number between translation distances, etc.

The paper is organized as follows.
In Section \ref{sec.horo}, we review basic facts of horofunctions.
Then in Section \ref{sec.conti}, we prove the continuity of the drift $\ell(\mu)$, under the assumption that the distance and  the horofunctions are written as certain ratios.
In Section \ref{sec.ex}, we discuss $\mcg$ and prove the continuity of the drift with respect to several distances.
In Section \ref{sec.NS}, degenerating sequences are discussed.
Finally, in the appendix, we prove the continuity of the asymptotic entropy in $\mcg$.

\section{Horofunction boundary}\label{sec.horo}
Let $(X,d)$ be a separable possibly asymmetric metric space.
A function $f:X\rightarrow \mathbb{R}$ is said to be {\em $1$-Lipschitz} if $|f(x)-f(y)|\leq d(x,y)$ for all $x,y\in X$.
Throughout the section, we fix a base point $b\in X$.
We define $$\Lip(X):=\{f:X\rightarrow\mathbb{R}\mid f \text{ is }1\text{-Lipschitz and }f(b)=0\}.$$
We consider the topology on $\Lip(X)$ by uniform convergence on compact sets.
\begin{prop}[see e.g. {\cite[Proposition 3.1]{MT}}]\label{prop.Lip-compact}
Let $(X,d)$ be a separable possibly asymmetric metric space.
Then $\Lip(X)$ is a compact Hausdorff second countable (hence metrizable) space.
\end{prop}
Using horofunctions which we now define, we embed $X$ into $\Lip(X)$.
\begin{defi}
A {\em horofunction} determined by $z\in X$ is a function $\psi_{z}:X\rightarrow\mathbb{R}$ with
$\psi_{z}(x) = d(x,z)-d(b,z)$ for any $x\in X$.
\end{defi}
By the triangle inequality, we see that $\psi_{z}\in\Lip(X)$ for every $z\in X$.
We define symmetrised  metric $d_{\mathrm{sym}}$ by $d_\mathrm{sym}(x,y) = d(x,y)+d(y,x)$.
\begin{lem}[{\cite[Proposition 2.1 and 2.2]{Wal}}]\label{lem.horo}
Let $(X,d)$ be a geodesic separable possibly asymmetric metric space.
The map $\psi:X\rightarrow \Lip(X)$ defined by $\psi(z):=\psi_{z}$ is continuous and injective.
Furthermore, if 
\begin{itemize}
\item $(X,d_\mathrm{sym})$ is a proper space (i.e. every closed metric ball is compact) and 
\item for any $x$ and sequence $x_{n}$, we have $d(x_{n},x) \rightarrow 0$ if and only if $d(x, x_{n}) \rightarrow 0$
\end{itemize}
then $\psi$ is a homeomorphism onto its image.
\end{lem}
By Proposition \ref{prop.Lip-compact} and Lemma \ref{lem.horo}, the closure $\overline{\psi(X)}$ of $\psi(X)$ in $\Lip(X)$ is compact.
The space $\partial_{h}X:=\overline{\psi(X)}\setminus\psi(X)$ is called the {\em horoboundary} of $X$.
From now on, we suppose that $X$ is proper and geodesic.
\begin{notation*}
By abuse of notations, we write $\overline{\psi(X)}$ by $X\cup\hb X$.
For $z\in X\cup\hb X$, we write the associated horofunction by $\psi_{z}$.
\end{notation*}

Let $G$ be a group acting on $X$ by isometries.

\begin{lem}[{\cite[Lemma 3.4]{MT}}]\label{lem.action}
Let $G$ be a group of isometries of $X$. 
Then the action of $G$ on $X$ extends to a continuous action by homeomorphisms on $\overline{\psi(X)}$, defined as
$$g\cdot \psi(z):=\psi(g^{-1}z)-\psi(g^{-1}b)$$
for each $g\in G$ and $\xi\in\overline{\psi(X)}$.
\end{lem}

A function $c:G\times \partial_{h}X\rightarrow \mathbb{R}$ is called a {\em cocycle} if it satisfies the cocycle property:
$$c(gh,{\xi}) = c(g,h\cdot {\xi})+c(h,{\xi}),$$
for every $g,h\in G$ and $\psi_{\xi}\in\partial_{h}X$.
The {\em Busemann cocycle} plays a key role in this paper.
\begin{defi}
The {\em Busemann cocycle} is a cocycle defined by $$c_{B}(g,{\xi}) = \psi_{\xi}(g^{-1}b).$$
\end{defi}


\section{Continuity of drift}\label{sec.conti}
\subsection{Compare distance and horofunctions}
Let $(X,d)$ be a connected, proper, separable, geodesic, possibly asymmetric metric space on which a countable group $G$ acts by isometries.
If the distance is asymmetric, we also suppose that for any $x\in X$ and sequence $x_{n}$, 
$d(x_{n},x) \rightarrow 0$ if and only if $d(x, x_{n}) \rightarrow 0$ so that Lemma \ref{lem.horo} applies.
In this section, we consider the case where the following two conditions hold.
\begin{cond}[Distance is characterized by a ratio]\label{cond.dist}
There are a set $\mathcal{S}$ and a  map
$X\rightarrow \mathbb{R}_{>0}^{\mathcal{S}}$ denoted by $x\mapsto i(x,\cdot)$ such that 
for each $x\in X$, $\inf_{\alpha\in\mathcal{S}} i(x,\alpha)>0$, and
for any $x,y\in X$,
\begin{equation}
d(x,y)=\log\sup_{\alpha\in\mathcal{S}}\frac{i(y,\alpha)}{i(x,\alpha)}.\label{eq.dist}
\end{equation}
\end{cond}
\begin{notation*}
By equation (\ref{eq.dist}), for any $x,y\in X$ and $\epsilon>0$, there  exists $\alpha_{(x,y,\epsilon)}\in\mathcal{S}$
such that 
\begin{equation}
d(x,y)\leq \log\frac{i(y,\alpha_{(x,y,\epsilon)})}{i(x,\alpha_{(x,y,\epsilon)})}+\epsilon.\label{eq.epsilon}
\end{equation}
There can be several elements in $\mathcal{S}$ that satisfy (\ref{eq.epsilon}),
and by $\alpha_{(x,y,\epsilon)}$, we denote one of them.
\end{notation*}

Let $\PRs$ denote the quotient $\mathbb{R}_{\geq 0}^{\mathcal{S}}/\mathbb{R}_{>0}$ with respect to the diagonal action.
\begin{cond}[Horofunction is characterized by a ratio]\label{cond.horo}
We can extend $i(x,\cdot)$ to $\partial_{h}X$ projectively i.e. 
$\hb X\ni \xi\mapsto i(\xi,\cdot)\in \PRs$, so that 
every horofunction is characterized in terms of a ratio:
for any $\xi\in X\cup\hb X$,
\begin{equation}
\psi_{\xi}(x) = \log\left(\sup_{\alpha\in\mathcal{S}}\frac{i(\xi,\alpha)}{i(x,\alpha)} 
\bigg/\sup_{\alpha\in\mathcal{S}}\frac{i(\xi,\alpha)}{i(b,\alpha)}\right).\label{eq.horo}
\end{equation}
Here the right-hand side is well-defined as we have $i(\xi,\cdot)$ both in the numerator and the denominator.
\end{cond}

In Section \ref{sec.ex}, we consider Thurston's Lipschitz distance and the Teichm\"uller distance on the Teichm\"uller space.
Those two distances satisfy Condition \ref{cond.dist} and Condition \ref{cond.horo}.

We compare horofunctions and the distance as follows.
\begin{lem}\label{lem.compare}
Suppose that $(X,d)$ satisfies Condition \ref{cond.dist} and \ref{cond.horo}.
Let $\xi\in X\cup\hb X$.
We fix a representative $i(\xi,\cdot)$ in $\mathbb{R}_{\geq 0}^{\mathcal{S}}$.
Suppose that a sequence $x = (x_{n})_{n\in\mathbb{Z}_{+}}$ of elements in $G$ satisfies the following condition.
\begin{cond}\label{cond.seq}
There exist $\epsilon>0$ and $C'>0$ such that $\alpha_{n}:=\alpha_{(x_{n},b,\epsilon)}$ satisfies
$$\frac{i(\xi,\alpha_{n})}{i(b,\alpha_{n})}\geq C',$$
for all $n\in\mathbb{Z}_{+}$.
\end{cond}
Then there exists a constant $C = C(\xi,\omega,C',\epsilon)$ 
which is independent of $n$ such that
\begin{equation}
d(x_{n},b)\geq \psi_{\xi}(x_{n})\geq d(x_{n},b) - C.\label{eq.compare}
\end{equation}
In particular, we have
$$\lim_{n\rightarrow\infty} \frac{1}{n}\psi_{\xi}(x_{n}) = \lim_{n\rightarrow\infty}\frac{1}{n}d(x_{n},b).$$
\end{lem}
\begin{proof}
First, by the triangle inequality we always have
$$\psi_{\xi}(y)\leq d(y,b).$$
We let 
$C'':=\log \sup_{\alpha\in\mathcal{S}} {i(\xi,\alpha)}/{i(b,\alpha)}$, 
which is independent of $n$.
Then by Condition \ref{cond.dist}, \ref{cond.horo}, and \ref{cond.seq},
\begin{align*}
\psi_{\xi}(x_{n}) &=  \log \sup_{\alpha\in\mathcal{S}} \frac{i(\xi,\alpha)}{i(x_{n},\alpha)}- C''\\
&\geq \log \frac{i(\xi,\alpha_{n})}{i(x_{n},\alpha_{n})} - C''\\
&= \log\frac{i(b,\alpha_{n})}{i(x_{n},\alpha_{n})} + \log\frac{i(\xi,\alpha_{n})}{i(b,\alpha_{n})}-C''\\
&\geq d(x_{n},b) - \epsilon + \log\frac{i(\xi,\alpha_{n})}{i(b,\alpha_{n})} - C''\\
&\geq d(x_{n},b) - \epsilon + \log C' -C''
\end{align*}
Set $C:=  \epsilon - \log C' +C''$.
As $\log C'\leq C''$ by definition, $C>0$.
Note that $C$ is independent of the choice of the representative $i(\xi,\cdot)\in\mathbb{R}_{\geq 0}^{\mathcal{S}}$.
Thus we get (\ref{eq.compare}).
The equality
$$\lim_{n\rightarrow\infty} \frac{1}{n}\psi_{\xi}(x_{n}) = \lim_{n\rightarrow\infty}\frac{1}{n}d(x_{n},b).$$
immediately follows from (\ref{eq.compare}).
\end{proof}
\subsection{Integral formula and continuity of drift}
We now consider (right) random walks on $G$.
Let $\mu$ be a probability measure on $G$ and $G^{\mathbb{Z}_{+}}$ the space of sample paths.
A cylinder is a subset of $G^{\mathbb{Z}_{+}}$ defined by 
$$[x_{1},\cdots, x_{n}] = \{\omega = (\omega_{i})_{i\in\mathbb{Z}_{+}}\in\G^{\mathbb{Z}_{+}}\mid \omega_{i} = x_{i}\text{ for } 1\leq i\leq n\}.$$
The probability measure $\mu$ induces a probability measure $\mathbb{P}$ on $G^{\mathbb{Z}_{+}}$, which is characterized by
$$\mathbb{P}([x_{1},\cdots,x_{n}]) = \mu(x_{1})\mu(x_{1}^{-1}x_{2})\cdots \mu(x_{n-1}^{-1}x_{n}).$$

Also let $\breve\mu$ denote the {\em reflected measure} of $\mu$ defined by $\breve\mu(g):=\mu(g^{-1})$ and 
$\breve{\mathbb{P}}$ the induced probability measure on $G^{\mathbb{Z}_{-}}$.
We let
\begin{equation}
L(\mu):=\sum_{g\in G}\mu(g)d(b,gb), \text{ and } \breve L(\mu):=\sum_{g\in G}\mu(g)d(gb,b).\label{eq.L}
\end{equation}
If $L(\mu)$ and $\breve L(\mu)$ are finite, then $\mu$ is said to have {\em finite first moment} with respect to the distance $d$.
By Kingman's subadditive ergodic theorem, if $\mu$ has finite first moment,
there exists $\ell(\mu)\geq 0$ such that for $\mathbb{P}$-a.e. $\omega = (\omega_{n})\in G^{\mathbb{Z}_{+}}$,
$$\lim_{n\rightarrow\infty}\frac{1}{n}d(\omega_{n}b,b) = \ell(\mu).$$
This $\ell(\mu)$ is called the {\em drift} of the random walk with transition probability $\mu$ with respect to the distance $d$.

We equip $\hb X$ the subset topology in $\Lip(X)$, and consider the Borel $\sigma$-algebra.
A measure $\nu$ on $\hb X$ is called {\em $\mu$-stationary}
if $$\nu(A)=\sum_{g\in G}\mu(g)\nu(g^{-1}A)$$
for every measurable subset $A\subset \hb X$.
Since $\hb X$ is compact, $\hb X$ always admits a $\mu$-stationary measure (see e.g. \cite[Lemma 2.2.1]{KM}).

First, we prove
\begin{lem}[c.f.{\cite[Proposition 2.2]{GMM}}]\label{lem.drift}
Let $\mu$ be a probability measure on $G$ with finite first moment.
Also let $\nu$ (resp. $\breve\nu$) be a $\mu$-stationary (resp. $\breve\mu$-stationary) measure on $\hb X$.
Suppose that for ${\mathbb{P}}$-a.e. $\omega = (\omega_{n})_{n\in\mathbb{Z}_{+}}\in G^{\mathbb{Z}_{+}}$ and $\breve\nu$-a.e. $\xi$,
the sequence $(\omega_{n}b)_{n\in\mathbb{Z}_{+}}$ and $\xi$ satisfy Condition \ref{cond.seq}.
Then 
$$\ell(\mu) = \int
c_{B}(g,\xi)d\breve\mu(g)d\breve\nu(\xi).$$

%

\end{lem}
\begin{proof}
The proof goes similarly to \cite[Proposition 2.2]{GMM}.
By the cocycle property,
\begin{align*}
&\int c_{B}(g_{n}\cdots g_{1},\xi)d\breve\mu(g_{1})\cdots d\breve\mu(g_{n})d\breve\nu(\xi)\\
= &\sum_{k=1}^{n} \int c_{B}(g_{k},g_{k-1} \cdots g_{1}\xi)d\breve\mu(g_{1})\cdots d\breve\mu(g_{k})d\breve\nu(\xi)
\end{align*}
Since the measure $\breve\nu$ is $\breve\mu$-stationary, the point $g_{k-1}\cdots g_{1}\xi$ is distributed according to $\breve\nu$.
Hence, the terms in the above sum do not depend on $k$. Thus we get
$$\int c_{B}(g,\xi)d\breve\mu(g)d\breve\nu(\xi) = \frac{1}{n} \int c_{B}(\omega_{n}^{-1},\xi)d\mathbb{P}(\omega)d\breve\nu(\xi).$$
By Lemma \ref{lem.compare} and our assumption, 
we have $$\lim_{n\rightarrow\infty} c_{B}(\omega_{n}^{-1},\xi)/n = \lim_{n\rightarrow\infty}\psi_{p(\xi)}(\omega_{n}b)/n= \ell(\mu)$$
for $\mathbb P$-a.e $\omega$ and $\breve\nu$-a.e. $\xi$.
Note that by the triangle inequality $|c_{B}(\omega_{n}^{-1},\xi)/n| = |\psi_{p(\xi)}(\omega_{n}b)/n|\leq \max(d(\omega_{n}b,b)/n, d(b, \omega_{n}b)/n)$.
Since $\mu$ has finite first moment, $d(\omega_{n}b,b)/n$ and $d(b,\omega_{n}b)/n$ are integrable.
By the bounded convergence theorem, the conclusion holds.
\end{proof}

The goal of this section is to prove the following.

\begin{thm}[c.f. {\cite[Proposition 2.3]{GMM}}]\label{thm.conti}
Let $X$ and $G$ be a metric space and a group as in the beginning of this section.
Suppose that $X$ satisfies Condition \ref{cond.dist} and \ref{cond.horo}.
Let $(\mu_{i})_{i\in\mathbb{N}}$ be a sequence of probability measures on $G$ with finite first moment, converging simply to a probability measure $\mu_{\infty}$ (i.e. $\mu_{i}(g)\rightarrow \mu_{\infty}(g)$ for all $g\in G$).
For each $i\in\mathbb{N}\cup\{\infty\}$, 
let $\mathbb{P}_{i}$ (resp. $\breve{\mathbb{P}_{i}}$) denote the probability measure on 
$G^{\mathbb{Z}_{+}}$ (resp. $G^{\mathbb{Z}_{-}}$) induced by $\mu_{i}$  (resp. $\breve\mu_{i}$).
Let further $\nu_{i}$ (resp. $\breve\nu_{i}$) be 
a $\mu_{i}$-stationary  (resp. $\breve\mu_{i}$-stationary) measure on $\hb X$.
We suppose that $\nu_{i}$ (resp. $\breve\nu_{i}$) converges weakly to $\nu_{\infty}$ (resp. $\breve\nu_{\infty}$).
Suppose further that 
for every $i\in\mathbb{N}\cup\{\infty\}$, 
$\breve{\mathbb{P}_{i}}$-a.e. $\omega=(\omega_{n})_{n\in\mathbb{Z}_{-}}$ and $\breve\nu_{i}$-a.e. $\xi$,
the sequence $(\omega_{n}b)_{n\in\mathbb{Z}_{-}}$ and $\xi$ satisfy Condition \ref{cond.seq}, 
and $\breve L(\mu_{i})\rightarrow \breve L(\mu_{\infty})$.


Then $\ell(\mu_{i}) \rightarrow \ell(\mu_{\infty})$. 
\end{thm}
\begin{proof}
The proof goes similarly to \cite[Proposition 2.3]{GMM}.
For every $g\in G$, $\int c_{B}(g,\xi)d\breve\nu_{i}(\xi)$ converges to $\int c_{B}(g,\xi)d\breve\nu_{\infty}(\xi)$ as 
$|c_{B}(g,\xi)|\leq d(g^{-1}b,b)=d(b,gb)$ and $c_{B}$ is continuous.
Then since $\breve L(\mu_{i})\rightarrow \breve L(\mu_{\infty})$, by Lemma \ref{lem.drift} we have
$$\ell(\mu_{i}) = 
\sum_{g\in G}\breve\mu_{i}(g)\int c_{B}(g,\xi)d\breve\nu_{i}(\xi)\rightarrow 
\sum_{g\in G}\breve\mu_{\infty}(g)\int c_{B}(g,\xi)d\breve\nu_{\infty}(\xi) 
= \ell(\mu_{\infty}).
$$
\end{proof}

\section{Application to the mapping class group}\label{sec.ex}
In this section, we apply Theorem \ref{thm.conti} to several distances related to the mapping class group $\mcg$.

\subsection{Thurston's Lipschitz distance}\label{sec.Thur}
Let $S$ be an orientable surface of finite type.
Let $g$ be the genus, and $p$ the number of punctures of $S$.
Throughout the paper, we suppose that $3g+p-3>0$.
Let $\mathcal{T}(S)$ denote the Teichm\"uller space of $S$.
In this subsection, we regard the Teichm\"uller space $\mathcal{T}(S)$ as the space of marked hyperbolic structures on $S$ up to certain equivalence relation.

Thurston's Lipschitz distance of two points $x,y$ in $\mathcal{T}(S)$ is defined as
$$d_{\mathrm{Th}}(x,y):=\inf_{\phi:x\rightarrow y} \log L(\phi), $$
where $\phi:x\rightarrow y$ is a homeomorphism with Lipschitz constant $L(\phi)$ and 
the infimum is taken over all homeomorphisms preserving the markings.
This is an asymmetric distance.
Let $\mathcal{S}$, $\mf$, and $\pmf$ denote the set of all isotopy classes of essential simple closed curves on $S$,
the space of measured foliations, and the space of projective measured foliations respectively.
By the work of Thurston \cite{Thur}, we have
\begin{equation}
d_{\mathrm{Th}}(x,y)=\log\sup_{\alpha\in\mathcal{S}}\frac{\mathrm{Len}_{y}(\alpha)}{\mathrm{Len}_{x}(\alpha)},\label{eq.Tdist}
\end{equation}
where for an isotopy class $c$ of simple closed curves, $\mathrm{Len}_{x}(c)$ denote the hyperbolic length of the shortest representative in $c$.

Let $\PRs$ denote the quotient $\mathbb{R}_{\geq 0}^{\mathcal{S}}/\mathbb{R}_{>0}$ with respect to the diagonal action.
Thurston compactified $\mathcal{T}(S)$ by embedding it into $\PRs$ by 
$$\mathcal{T}(S)\ni x\mapsto \mathrm{Len}_{x}(\cdot)\in \PRs.$$
Let $\overline{\mathcal{T}(S)}^{\mathrm{Th}}$ denote the closure of $\mathcal{T}(S)$ in $\PRs$.
Thurston showed that $\overline{\mathcal{T}(S)}^{\mathrm{Th}}\setminus \mathcal{T}(S)$
can be identified with $\pmf$, which contains $\mathcal{S}$ as a dense subset.
By letting $i(x,\cdot) = \mathrm{Len}_{x}(\cdot)$ if $x\in \mathcal{T}(S)$ and $i(x,\cdot)$ the intersection number if $x\in\pmf$,
we have a continuous function $\overline{\mathcal{T}(S)}^{\mathrm{Th}}\rightarrow \PRs$.
Hence the suprema in (\ref{eq.Tdist}) and (\ref{eq.Thoro}) in the below are attained in $\pmf$.
See e.g. \cite{FLP} for a detail.

Let $\psi_{z}^{\mathrm{Th}}$ denote the horofunction corresponding to $z\in\mathcal{T}(S)$ 
with respect to the distance $d_{\mathrm{Th}}$.
In \cite{Wal}, Walsh identified the horoboundary $\partial_{h}\mathcal{T}(S)$ with Thurston's boundary $\pmf$.
Furthermore, Walsh \cite{Wal} showed that for any $\xi\in\pmf$, the horofunction $\psi_{\xi}^{\mathrm{Th}}$ can be expressed as
\begin{equation}
\psi_{\xi}^{\mathrm{Th}}(x) = \log\left(\sup_{\eta\in\mathcal{S}}\frac{i(\xi,\eta)}{i(x,\eta)}\bigg/\sup_{\eta\in\mathcal{S}}\frac{i(\xi,\eta)}{i(b,\eta)}\right),\label{eq.Thoro}
\end{equation}
for $x\in \mathcal{T}(S)$.
Note that the right-hand side of (\ref{eq.Thoro}) is independent of the choice of representatives of $\xi$ in $\mf$.

As $i|_{\mathcal{T}(S)\times\mathcal{S}}$ is defined by $\mathrm{Len}(\cdot)$, equations (\ref{eq.Tdist}) and (\ref{eq.Thoro}) show that  Thurston's asymmetric  Lipschitz distance satisfies Condition \ref{cond.dist} and \ref{cond.horo}.
Therefore our goal in this subsection is to confirm Condition \ref{cond.seq} almost surely.

We may identify $\pmf$ with $\mathcal{P}:=\{F\in\mf\mid i(b,F)=1\}$.
Then we have the following.
\begin{lem}\label{lem.Tseq}
Let $x = (x_{n})_{n\in\mathbb{Z}_{+}}$ be a sequence in $\mathcal{T}(S)$ which converges to $F(x)\in \pmf$ as $n\rightarrow\infty$ in Thurston's compactification.
If $\xi\in\mf$ satisfies that there exists $C>0$ such that $i(\xi,\eta)\geq C>0$ for all $\eta\in\mathcal{P}$ with $i(F(x),\eta)=0$,
then Condition \ref{cond.seq} holds for $x = (x_{n})_{n\in\mathbb{Z}_{+}}$ and $\xi$.
\end{lem}
\begin{proof}
Let $\alpha_{n}\in \mathcal{P}$ be an element which satisfies $d_{\mathrm{Th}}(x_{n},b)=\log (1/i(x_{n},\alpha_{n}))$.
By taking subsequence if necessary, we may suppose $\alpha_{n}\rightarrow\alpha_{\infty}\in\mathcal{P}$ ($\mathcal{P}$ is compact).
We take the representative of $F(\omega)$ in $\mathcal{P}$.
As $x_{n}$ converges to $F(x)\in\mathcal{P}$ in Thurston's compactification, there exist $\lambda_{n}$ such that
$\lambda_{n}\rightarrow 0$ and $\lambda_{n}x_{n}\rightarrow F(x)$ in $\mathbb{R}_{\geq 0}^{\mathcal{S}}$.
Since $$d_{\mathrm{Th}}(x_{n},b) = \log\frac{\lambda_{n}}{i(\lambda_{n}x_{n},\alpha_{n})}\rightarrow \infty,$$
$i(\lambda_{n}x_{n},\alpha_{n})$  must go to $0$.
Hence we have $i(F(x),\alpha_{\infty})=0$ and so by our assumption,
$i(\xi,\alpha_{\infty})\geq C>0$.
Then the result follows from the continuity of $i(\cdot,\cdot)$ and the density of $\mathcal{S}\subset\mathcal{P}$.
\end{proof}

Let $\mcg$ denote the mapping class group, namely 
the group of all isotopy classes of orientation preserving homeomorphisms on $S$.
We now consider random walks on $\mcg$.
As we assumed $3g+p-3>0$, $\mcg$ contains pseudo-Anosov elements.
Two pseudo-Anosov elements are called {\em independent} if their fixed point sets in $\pmf$ are disjoint.
A probability measure $\mu$ on $\mcg$ is called {\em non-elementary}
if the group generated by the support of $\mu$ contains two independent pseudo-Anosovs.
Let $\mathbb{P}$ denote the induced probability measure on $\mcg^{\mathbb{Z}_{+}}$.
Let us recall the work of Kaimanovich-Masur.
Let $\mathcal{UE}(S)\subset \mathcal{PMF}(S)$ denote the space of uniquely ergodic foliations.
For later use (in Appendix), we also mention the notion of $\mu$-boundary, see \cite{Kai, KM} for the definition and properties.
\begin{thm}[{\cite[Theorem 2.2.4]{KM}}]\label{thm.KM}
Let $\mu$ be a non-elementary probability measure on $\mcg$.
 Then the following holds.
\begin{enumerate}
 \item  There exists a unique $\mu$-stationary probability measure $\nu$
	on $\mathcal{PMF}(S)$ 
	which is purely non-atomic and concentrated on $\mathcal{UE}(S)$, and the measure space $(\mathcal{UE}(S), \nu)$ is a $\mu$-boundary.
\item  For $\mathbb{P}$-a.e. $\omega\in\mathrm{MCG}(S)^{\mathbb{Z}_{+}}$ and any $x\in\mathcal{T}(S)$,
       the sequence $\omega_nx$ converges in $\mathcal{PMF}(S)$ to a limit 
       $F(\omega)\in\mathcal{UE}(S)$ 
       and the distribution of the limits is given by the measure $\nu$.
\end{enumerate}
\end{thm}

As a corollary of Lemma \ref{lem.Tseq} and Theorem \ref{thm.KM}, we have the following.
\begin{coro}\label{cor.Tseq}
For $\mathbb{P}$-a.e. $\omega = (\omega_{n})_{n\in\mathbb{Z}_{+}}\in\mcg^{\mathbb{Z}_{+}}$ and 
$\breve \nu$-a.e $\xi$, $(\omega_{n}b)_{n\in\mathbb{Z}_{+}}$ and $\xi$ satisfy
Condition \ref{cond.seq}.
\end{coro}
\begin{proof}
By Theorem \ref{thm.KM}, we may suppose that $\xi$ and $F(\omega)$ are uniquely ergodic.
Then for $\eta\in\{\xi,F(\omega)\}$, we have that $i(\alpha,\eta) = 0$ implies $\alpha = \eta$ in $\pmf$ for any $\alpha\in\mf$.
Furthermore, as $\breve\nu$ is non-atomic, we may suppose that $i(F(\omega),\xi)=C>0$.
Hence the assumption of Lemma \ref{lem.Tseq} is satisfied for $(\omega_{n}b)_{n\in\mathbb{Z}_{+}}$ and $\xi$.
\end{proof}

We now define $c_{B}^{\mathrm{Th}}:\mcg\times(\mathcal{T}(S)\cup\pmf)\rightarrow\mathbb{R}$ by
 $$c_{B}^{\mathrm{Th}}(g,\xi):=\psi_{\xi}^{\mathrm{Th}}(g^{-1}b).$$
Also let $\ell_{\mathrm{Th}}(\mu)$ denote the drift of the random walk given by $\mu$ with respect to Thurston's Lipschitz distance and define $L_\mathrm{Th}(\mu)$ and $\breve L_\mathrm{Th}(\mu)$ using (\ref{eq.L}).
Then by Lemma \ref{lem.drift} and Corollary \ref{cor.Tseq} we have
\begin{prop}[c.f.{\cite[Proposition 2.2]{GMM}}]\label{prop.Tdrift}
Let $\mu$ be a non-elementary probability measure on $\mcg$ and $\breve\nu$ be a $\breve\mu$-stationary measure on $\pmf$.
Then 
$$\ell_{\mathrm{Th}}(\mu) = \int_{\mcg\times\pmf}c_{B}^{\mathrm{Th}}(g,\xi)d\breve\mu(g)d\breve\nu(\xi).$$
\end{prop}

%

Putting all the discussion in this subsection together, we have the continuity of the drift.

\begin{thm}\label{thm.conti-Thurston}
Let $(\mu_{i})_{i\in\mathbb{N}}$ be a sequence of non-elementary probability measures on $\mcg$, each of which has finite first moment with respect to Thurston's Lipschitz distance on $\mathcal{T}(S)$.
Suppose that $(\mu_{i})_{i\in\mathbb{N}}$ converges simply to a non-elementary
probability measure $\mu_{\infty}$ (i.e. $\mu_{i}(g)\rightarrow \mu_{\infty}(g)$ for all $g\in \mcg$).
Suppose further that $\breve L_\mathrm{Th}(\mu_{i})\rightarrow\breve L_\mathrm{Th}(\mu_{\infty})$.
Then $\ell_{\mathrm{Th}}(\mu_{i}) \rightarrow \ell_{\mathrm{Th}}(\mu_{\infty})$. 
\end{thm}
\begin{proof}
By Theorem \ref{thm.KM}, for each $i\in\mathbb{N}\cup\{\infty\}$, there is a unique $\breve\mu_{i}$-stationary measure $\nu_{i}$ on $\pmf$.
As the action of $\mcg$ on $\pmf$ is continuous, after taking subsequence, $\nu_{i}$ converges weakly to a $\mu_{\infty}$-stationary measure $\nu'$, and the uniqueness implies that $\nu'=\nu_{\infty}$.
Note that we have already seen that Condition \ref{cond.dist} and \ref{cond.horo} are satisfied for Thurston's Lipschitz distance.
By Corollary \ref{cor.Tseq}, we see that the assumptions of Theorem \ref{thm.conti} are all satisfied and the result follows.
\end{proof}

\subsection{ Teichm\"uller distance on $\mathcal{T}(S)$}\label{sec.Teich}
In this subsection, we consider the Teichm\"uller distance.
First, note that by the work of Choi-Rafi \cite{CR}, 
the Teichm\"uller distance and Thurston's Lipschitz distance differ at most by a constant amount in the thick part of $\mathcal{T}(S)$.
This fact implies that the drifts with respect to the Teichm\"uller distance and Thurston's Lipschitz distance coincide (see e.g. \cite{Mas}).
Hence the continuity of the drift with respect to the Teichm\"uller distance follows immediately from Theorem \ref{thm.conti-Thurston}.
However, our strategy in Section 3 also works for the Teichm\"uller distance and combining with the work of Choi-Rafi, 
we get an interesting stochastic relationship between the hyperbolic length and the extremal length.
Hence we prove the continuity of drift with respect to the Teichm\"uller distance without using Theorem \ref{thm.conti-Thurston}.
In this subsection, we regard the Teichm\"uller space $\mathcal{T}(S)$ as the space of equivalence classes of marked Riemann surfaces.
The Teichm\"uller distance is defined by
$$d_{\mathcal{T}}(x,y) = \frac{1}{2}\inf_{\phi:x\rightarrow y}\log K(\phi)$$
where $K(\phi)$ is the dilatation of a quasi-conformal map $\phi$ and the infimum is taken over all quasi-conformal maps $\phi:x\rightarrow y$ compatible with the markings.

We first recall the notion of extremal length.
\begin{defi}[Extremal length]
Let $x\in\mathcal{T}(S)$ and $\alpha\in\mathcal{S}$, the extremal length of $\alpha$ on $x$ is
$$\ext_{x}(\alpha):=\sup_{\sigma}\frac{L_{\sigma}^{2}(\alpha)}{A(\sigma)},$$
where the supremum is taken over all conformal metrics $\sigma(z)|dz|$ in $x$, $L_{\sigma}$ is the length function 
$$L_{\sigma}(\alpha):=\inf_{a\in\alpha}\int_{a}\sigma(z)|dz|,$$
and $A(\sigma)$ is the area
$$A(\sigma):=\int_{x} \sigma^{2}(z)|dz|^{2}.$$
\end{defi}

A key ingredient in this subsection is Kerckhoff's formula of the Teichm\"uller distance (see \cite{Ker}):
\begin{equation}\label{eq.tei-dist}
d_{\mathcal{T}}(x,y)=\frac{1}{2}\log\sup_{\alpha\in\mathcal{S}}\frac{\ext_{y}(\alpha)}{\ext_{x}(\alpha)}.
\end{equation}

Gardiner-Masur \cite{GM} showed that 
$$\mathcal{T}(S)\ni x\mapsto \ext_{x}(\cdot)^{1/2}\in \PRs$$
is an embedding with compact closure.
This is called the Gardiner-Masur compactification and we denote it by $\overline{\mathcal{T}(S)}^{GM}$.
The boundary $\partial_{GM}\mathcal{T}(S)$ of $\overline{\mathcal{T}(S)}^{GM}$ is called the Gardiner-Masur boundary.


For $x\in\mathcal{T}(S)$, let $K_{x}$ be a quantity given by
$d_{\mathcal{T}}(x, b) = \frac{1}{2}\log K_{x}$.
Miyachi \cite{Miy} introduced a function:
$$\mathcal{E}_{x}(\alpha) = \frac{\ext_{x}(\alpha)^{1/2}}{K_{x}^{1/2}}$$
for $x\in\mathcal{T}(S)$ and $\alpha\in\mf$.
Miyachi showed that the function $\mathcal{E}$ extends continuously to the Gardiner-Masur boundary.
\begin{lem}[\cite{Miy}]\label{lem.Miyachi}
For any $P\in\partial_{\mathrm{GM}}\mathcal{T}(S)$,
there is a non-negative continuous function $\mathcal{E}_{P} (\cdot)$ defined on $\mf$, such that
\begin{enumerate}
\item $\mathcal{E}_{P}(t\alpha) = t\mathcal{E}_{P}(\alpha)$ for $t>0$ and $\alpha\in\mf$, and
\item $\mathcal{E}_{P}:\mathcal{S}\rightarrow \mathbb{R}_{+}$ represents $P$ 
as an element of  $\PRs$.
\item If a sequence $(x_{n})_{n\in\mathbb{Z}_{+}}$ in $\mathcal{T}(S)$ converges to a point $P\in\partial_{\mathrm{GM}}\mathcal{T}(S)$, then after taking a subsequence if necessary, $\mathcal{E}_{x_{n}}(\cdot)$ converges to a positive multiple of $\mathcal{E}_{P}(\cdot)$ uniformly on compact subsets of $\mf$. Especially, we have
$$\lim_{n\rightarrow\infty}\frac{\ext_{x_{n}}(\alpha)^{1/2}}{\ext_{x_{n}}(\beta)^{1/2}} = \frac{\mathcal{E}_{P}(\alpha)}{\mathcal{E}_{P}(\beta)}$$
for all $\alpha,\beta\in\mathcal{MF}(S)$ with $\mathcal{E}_{P}(\beta)\not=0$.

\end{enumerate}
\end{lem}

By using the function $\mathcal{E}$, Liu-Su \cite{LS} showed that the horoboundary $\partial_{h}\mathcal{T}(S)$ of $\mathcal{T}(S)$
with respect to the Teichm\"uller distance is homeomorphic to the Gardiner-Masur boundary $\partial_{\mathrm{GM}}\mathcal{T}(S)$.
Let $\psi_{z}^{\mathcal{T}}$ denote the horofunction corresponding to $z\in\mathcal{T}(S)$ with respect to the Teichm\"uller distance.
Liu-Su \cite{LS} characterized the horofunction corresponding to $P\in\partial_{\mathrm{GM}}\mathcal{T}(S)$ as:
$$\psi_{P}^{\mathcal{T}}(x) = \log\left(\sup_{\alpha\in\mathcal{S}}\frac{\mathcal{E}_{P}(\alpha)}{\ext_{x}(\alpha)^{1/2}}\bigg/\sup_{\alpha\in\mathcal{S}}\frac{\mathcal{E}_{P}(\alpha)}{\ext_{b}(\alpha)^{1/2}}\right).$$

Now we define a function $\mathcal{E}:\overline{\mathcal{T}(S)}^{\mathrm{GM}}\times\mf\rightarrow \mathbb{R}_{+}$ by
\begin{equation*}
{\mathcal{E}(x,\alpha) := }
\begin{cases}
\ext_{x}(\alpha)^{1/2} & \text{ if } x\in\mathcal{T}(S)\\
\mathcal{E}_{x}(\alpha) & \text{ if } x\in\partial_{\mathrm{GM}}\mathcal{T}(S)\\
\end{cases}
\end{equation*}
With this $\mathcal{E}$, Condition \ref{cond.dist} and \ref{cond.horo} are satisfied for the Teichm\"uller distance.
We now ensure Condition \ref{cond.seq} almost surely.
First, we need a Teichm\"uller distance version of Lemma \ref{lem.Tseq}, which is already given in \cite{LS}.
\begin{lem}[{\cite[Lemma 4.5]{LS}}]\label{lem.formula}
Let $(y_{n})_{n\in\mathbb{Z}_{+}}$ be a sequence of points in $\mathcal{T}(S)$ converging to a point $P$ in the Gardiner-Masur boundary.
Let $x$ be a point in $\mathcal{T}(S)$.
Let $(\alpha_{n})_{n\in\mathbb{Z}_{+}}$ be a sequence in $\mathcal{PMF}(S)$ such that 
$$d_{\mathcal{T}}(y_{n},x) = \frac{1}{2}\log\frac{\mathrm{Ext}_{x}(\alpha_{n})}{\mathrm{Ext}_{y_{n}}(\alpha_{n})}.$$
For any $\beta\in\mathcal{MF}(S)$, if $\mathcal{E}_{P}(\beta) = 0$, 
then any limit point $\alpha_{\infty}\in\mathcal{PMF}(S)$ of a convergent subsequence of the sequence $(\alpha_{n})_{n\in\mathbb{Z}_{+}}$ satisfies $i(\beta,\alpha_{\infty}) = 0$.
\end{lem}

By the work of Gardiner-Masur \cite{GM}, $\partial_{\mathrm{GM}}\mathcal{T}(S)$ contains $\pmf$.
Furthermore Miyachi \cite[Corollary 1 in \S 6.1]{Miy-II} observed that a sequence points in 
$\mathcal{T}(S)$ converges to a uniquely ergodic foliation in Thurston's compactification
if and only if the sequence converges to the same uniquely ergodic foliation in the Gardiner-Masur compactification.
Recall that for $\mathbb{P}$-a.e. $\omega=(\omega_{n})_{n\in\mathbb{Z}_{+}}\in\mcg^{\mathbb{Z}_{+}}$, 
$\omega_{n}x$ converges to a point in $\mathcal{UE}(S)$.
This implies that $\omega_{n}x$ converges to a point in $\mathcal{UE}(S)$ in the Gardiner-Masur compactification as well.
Hence we may consider the $\mu$-stationary measure $\nu$ on $\partial_{\mathrm{GM}}\mathcal{T}(S)$ provided $\mu$ is non-elementary.

\begin{lem}\label{lem.GMseq}
For $\mathbb{P}$-a.e. $\omega=(\omega_{n})_{n\in\mathbb{Z}_{+}}\in\mcg^{\mathbb{Z}_{+}}$ and 
$\breve\nu$-a.e. $P\in\partial_{\mathrm{GM}}\mathcal{T}(S)$,
$(\omega_{n}b)_{n\in\mathbb{Z}_{+}}$ and $P$ satisfy Condition \ref{cond.seq}.
\end{lem}
\begin{proof}
By Theorem \ref{thm.KM}, we may suppose $F(\omega)$ and $P$ are uniquely ergodic in $\pmf$.
Also as $\breve\nu$ is non-atomic, we may suppose that $\mathcal{E}(P, F(\omega))>0$.
Then by Lemma \ref{lem.formula}, for the sequence $\alpha_{n}$ of points in $\mathcal{P}$ which satisfy
$$d_{\mathcal{T}}(\omega_{n}b,b) = \log\frac{\mathcal{E}(b,\alpha_{n})}{\mathcal{E}(\omega_{n}b,\alpha_{n})},$$
any limit point $\alpha_{\infty}$ satisfies $i(\alpha_{\infty},F(\omega)) = 0$, i.e. $\alpha_{\infty} = F(\omega)$ in $\mathcal{P}$.
As $\mathcal{E}$ is continuous, this implies $\mathcal{E}(P,\alpha_{n})\rightarrow \mathcal{E}(P,F(\omega))$, 
in particular, $\mathcal{E}(P,\alpha_{n})$ is bounded below by some constant independent of $n$ for large enough $n$.
Then Lemma \ref{lem.Miyachi} and the density of $\mathcal{S}$ in $\mathcal{P}$ imply Condition \ref{cond.seq} for $\omega$ and $P$.
\end{proof}

Now we define $c_{B}^{\mathcal{T}}:\mcg\times\overline{\mathcal{T}(S)}^{\mathrm{GM}}\rightarrow\mathbb{R}_{+}$ by
$$c_{B}^{\mathcal{T}}(g,P):=\psi_{P}^{\mathcal{T}}(g^{-1}b).$$
We denote by $\ell_{\mathcal{T}}(\mu)$ the drift of the random walk given by $\mu$ with respect to the Teichm\"uller distance and define $L_\mathcal{T}(\mu)$ and $\breve L_\mathcal{T}(\mu)$ using (\ref{eq.L}).

By Lemma \ref{lem.GMseq}, the assumption of Lemma \ref{lem.drift} is satisfied.
Thus we have,
\begin{prop}[{c.f. \cite[Proposition 2.2]{GMM}}]\label{prop.teich-drift}
Let $\mu$ be a non-elementary probability measure on $\mathrm{MCG}(S)$ with finite first moment with respect to the Teichm\"uller distance.
Let $\breve\nu$ be the $\breve\mu$-stationary measure on $\mathcal{PMF}(S)$.
Then 
$$\ell_{\mathcal{T}}(\mu) = \int_{\mcg\times\partial_{\mathrm{GM}}\mathcal{T}(S)} c_{B}^{\mathcal{T}}(g,\xi)d\breve\mu(g)d\breve\nu(\xi).$$
\end{prop}

Then we have the continuity of the drift with respect to the Teichm\"uller distance.
\begin{thm}\label{thm.conti-Teich}
Let $(\mu_{i})_{i\in\mathbb{N}}$ be a sequence of non-elementary probability measures on $\mcg$ with finite first moment   
with respect to the Teichm\"uller distance on $\mathcal{T}(S)$, 
converging simply to a non-elementary probability measure $\mu_{\infty}$.
Suppose that $\breve{L}_{\mathcal{T}}(\mu_{i})\rightarrow \breve{L}_{\mathcal{T}}(\mu_{\infty})$.
Then $\ell_{\mathcal{T}}(\mu_{i}) \rightarrow \ell_{\mathcal{T}}(\mu_{\infty})$. 
\end{thm}

As the proof can go similarly to the one for Theorem \ref{thm.conti-Thurston}, we omit the proof.

As we discussed at the beginning of this subsection, the work of Choi-Rafi \cite{CR} implies that 
$\ell_{\mathcal{T}}(\mu) = \ell_{\mathrm{Th}}(\mu)$.
Furthermore, we see from \cite{CR} that having finite first moment with respect to the Teichm\"uller distance and 
Thurston's Lipschitz distance are equivalent.
Putting this together with Proposition \ref{prop.Tdrift} and Proposition \ref{prop.teich-drift},
we have the following.
\begin{coro}\label{cor. hyp and ext}
Let $\mu$ be a non-elementary probability measure on $\mcg$ with finite first moment.
Then we have
\begin{align*}
&\int_{\mcg\times\pmf}
\log\left(\sup_{\alpha\in\mathcal{S}}\frac{i(\xi,\alpha)}{\mathrm{Len}_{g^{-1}b}(\alpha)} \bigg/ \sup_{\alpha\in\mathcal{S}}\frac{i(\xi,\alpha)}{\mathrm{Len}_{b}(\alpha)}d\mu(g)d\nu(\xi)\right)\\
=&\int_{\mcg\times\pmf}
\log\left(\sup_{\alpha\in\mathcal{S}}\frac{i(\xi,\alpha)}{\ext_{g^{-1}b}(\alpha)^{1/2}}\bigg/
\sup_{\alpha\in\mathcal{S}}\frac{i(\xi,\alpha)}{\ext_{b}(\alpha)^{1/2}}d\mu(g)d\nu(\xi)\right).
\end{align*}
\end{coro}
\begin{proof}
As the Gardiner-Masur boundary contains Thurston's boundary in $P\mathbb{R}^{\mathcal{S}}$,
we have $$\frac{i(\alpha,\beta)}{i(\alpha,\gamma)} = \frac{\mathcal{E}_{\alpha}(\beta)}{\mathcal{E}_{\alpha}(\gamma)},$$
if $\alpha\in\pmf$ and $\beta,\gamma\in\mf$.
Hence the results follows from the characterization of horofunctions and Proposition \ref{prop.Tdrift} and \ref{prop.teich-drift}.
\end{proof}

\section{Continuity of drift of degenerating sequence}\label{sec.NS}
Let $G$ be a group of isometries on a metric space $(X,d)$ with a compact, metrizable boundary $\partial X$.
We suppose that the action of $G$ extends continuously on $\partial X$.
We consider a probability measure $\mu$ on $G$ with finite first moment i.e. $L(\mu)<\infty$.
First note that the drift $\ell(\mu)$ satisfies $\ell(\mu)=\lim_{n\rightarrow\infty} L(\mu^{*n})/n$, and $L(\mu^{*n})\leq nL(\mu)$.
Hence we have
\begin{prop}\label{prop.zero}
If a sequence of probability measures $\mu_{i}$ converges to a probability measure $\mu_{\infty}$ with
$L(\mu_{i})\rightarrow L(\mu_{\infty})$ and 
$L(\mu_{\infty}) = 0$, then $\ell(\mu_{i})$ converges to $0$.
\end{prop}

For example, if $\mu_{\infty}=\delta_{\rm{id}}$, then $L(\mu_{\infty})=0$, where 
$\delta_{g}$ denotes the Dirac measure on $g\in G$.
More non-trivial case is when $\mu_{\infty}$ is $\delta_{a}$ where $a\in G$ has a fixed point on $X$.
As $\ell(\mu)$ is independent of the choice of the base point, we may choose the fixed point as the base point, and then
$L(\delta_{a}) = 0$.
The condition $L(\mu_{i})\rightarrow L(\mu_{\infty})$ is satisfied if, for example, 
$\mu_{i}$ has the same finite support for all $i\in\mathbb{N}$.

Another case we consider in this section is when $\mu_{\infty} = \delta_{g}$ where 
$g\in G$ acts on $\partial X$ with so called the north-south dynamics.
 An element $g\in G$ acting on $\partial X$ has {\em north-south dynamics} 
 if there are two fixed points $\gamma_{+},\gamma_{-}\in \partial X$ with the following property:
for any open neighborhood $U_{+}\subset  \partial X$ of $\gamma_{+}$ 
 (resp. $U_{-}\subset \partial X$ of $\gamma_{-}$) 
 and compact set $K_{+}\subset \partial X$ with $\gamma_{-}\not\in K_{+}$ 
 (resp. $K_{-}\subset \partial X$ with $\gamma_{+}\not\in K_{-}$),
  there exists $N\in\mathbb{N}$ such that
 for any $n\geq N$, we have $g^{n}(K_{+})\subset U_{+}$ (resp. $g^{-n}(K_{-})\subset U_{-}$).
 We define the translation distance of $g$ on $X$ by
 $$\ell(g):=\lim_{n\rightarrow\infty}\frac{ d(g^{n}b,b)}{n},$$
 where $b\in X$ is a fixed point.
When $\mu_{\infty}=\delta_{g}$, 
$\mu_{\infty}$-stationary measures are not unique.
As pointed out in \cite{GMM}, non-uniqueness of the stationary measures sometimes causes the non-continuity of the drift.
For example,  on the infinite dihedral group $\mathbb{Z}\rtimes \mathbb{Z}/2$,
 the measures $\mu_{i} = (1-1/i)\delta_{(1,0)} +(1/i )\delta_{(0,1)}$ 
  have zero drift since the $\mathbb{Z}/2$ element symmetrises everything in $\mathbb{Z}$,
   while the limiting measure $\mu_{\infty}=\delta_{(1,0)}$ has drift 1.
 By using north-south dynamics, we will prove that 
$\mu_{i}$-stationary measures $\nu_{i}$ converges to a ``nice'' $\mu_{\infty}$-stationary measure.
The following lemma is an easy consequence of $\mu_{\infty}$-stationarity.
\begin{lem}\label{lem.ns}
Let $g\in G$ be an element with north-south dynamics with fixed points $\gamma_{+},\gamma_{-}\in\partial X$.
Let further $\mu_{\infty}$ be a probability measure on $G$ defined by $\mu_{\infty}(g) = 1$.
Then any $\mu_{\infty}$-stationary measure $\nu_{\infty}$ is concentrated on the set $\{\gamma_{+},\gamma_{-}\}$, in other words
$\nu_{\infty} = p\delta_{\gamma_{+}}+(1-p)\delta_{\gamma_{-}}$ where $p\in[0,1]$ and $\delta_{\gamma_{+}},\delta_{\gamma_{-}}$ are the Dirac measures.
\end{lem}

\begin{proof}
Note that since $\partial X$ is metrizable, $\nu_{\infty}$ is regular.
As $\nu_{\infty}$ is $\mu_{\infty}$-stationary, 
$\nu_{\infty}(A) =\nu_{\infty}(g^{-n}A)$ for all $n\in\mathbb{N}$ and measurable subset $A\subset\partial X$.
By the north-south dynamics of $g$ if the closure $\overline A$ of $A$ does not contain $\gamma_{+}$ and $\gamma_{-}$, 
then there must be infinitely  many disjoint translates of $A$. Hence $\nu_{\infty}(A)=0$. 
As this holds for any $A$ with $\gamma_{+},\gamma_{-}\not\in\overline A$, the conclusion holds.
\end{proof}
Now we prove the following.
\begin{lem}\label{lem.ns2}
 Let $g\in G$ have north-south dynamics with fixed points $\gamma_{+}, \gamma_{-}\in\partial X$
  and $h\in G$ be an element such that 
  $\{h\gamma_{+}(g), h\gamma_{-}(g)\}\cap \{\gamma_{+}(g),\gamma_{-}(g)\} = \emptyset$.
 Let $\mu_{i}$ be a probability measure defined by $\mu_{i}(g) = 1-1/i$ and $\mu_{i}(h) = 1/i$, and $\mu_{\infty}$ defined by $\mu_{\infty}(g) = 1$.
 We further let $\nu_{i}$ be a $\mu_{i}$-stationary measure.
 Then any weak limit of $\nu_{i}$ is $\delta_{\gamma_{+}}$.
 \end{lem}
\begin{proof}
By continuity of the action of $G$ on $\partial X$, $\nu_{i}$ converges weekly to a $\mu_{\infty}$-stationary measure $\nu_{\infty}$.
Let $A_{-}$ be a small closed neighborhood of $\gamma_{-}$.
As $g$ has north-south dynamics, by taking $A_{-}$ small enough, 
we may suppose that $$A:=\bigcup_{n=0}^{\infty} g^{-n}(A_{-})$$ does not contain $\gamma_{+}$.
By our assumption of $h$, we may also assume that $h^{-1}A$ does not contain $\gamma_{+},\gamma_{-}$.
By construction, $g^{-1}A\subset A$,  and north-south dynamics of $g$ implies that $A$ is closed.
Then for $\nu_{i}$ is $\mu_{i}$-stationary, 
\begin{align*}
\nu_{i}(A) &= \mu_{i}(g)\nu_{i}(g^{-1}A) + \mu_{i}(h)\nu_{i}(h^{-1}A)\\
		     &\leq (1-1/i)\nu_{i}(A) + (1/i )\nu_{i}(h^{-1}A)\\
&\iff \nu_{i}(A)/i \leq \nu_{i}(h^{-1}A)/i\\
&\iff \nu_{i}(A)   \leq \nu_{i}(h^{-1}A).
\end{align*}
As $h^{-1}A$ misses both $\gamma_{+},\gamma_{-}$, Lemma \ref{lem.ns} implies that
$\nu_{i}(h^{-1}A)\rightarrow 0$ as $i\rightarrow \infty$.
Hence $\lim_{i\rightarrow\infty}\nu_{i}(A) = 0$.
This implies that for an open neighborhood $B\subset A$ of $\gamma_{-}$, we have
$\lim_{i\rightarrow\infty}\nu_{i}(B) = 0$.
As $B$ is open, 
$\nu_{\infty}(B)\leq \liminf_{i\rightarrow \infty}\nu_{i}(B) = 0$.
Again by Lemma \ref{lem.ns}, we see that $\nu_{\infty} = \delta_{\gamma_{+}}$.
\end{proof}

We now consider the mapping class group $\mcg$.
Recall that Thurston's boundary $\pmf$ of $\mathcal{T}(S)$ is homeomorphic to a sphere.
Any pseudo-Anosov element $g\in\mcg$ acts on $\pmf$ with north-south dynamics, 
and its fixed points $F_{+}(g)$ and $F_{-}(g)$ are uniquely ergodic.
 \begin{thm}\label{thm.ns}
 Let $g \in \mcg$ be a pseudo-Anosov element with fixed points $F_{+}(g), F_{-}(g)\in\pmf$
  and $h\in \mcg$ be such that $\{hF_{+}(g), hF_{-}(g)\}\cap \{F_{+}(g),F_{-}(g)\} = \emptyset$.
 Let $\mu_{i}$ be a probability measure defined by $\mu_{i}(g) = 1-1/i$ and $\mu_{i}(h) = 1/i$, and $\mu_{\infty}$ defined by $\mu_{\infty}(g) = 1$.
 Then $\ell(\mu_{i})$ converges to $\ell(g)$.
 \end{thm}

\begin{proof}
First, note that $\{hF_{+}(g), hF_{-}(g)\}\cap \{F_{+}(g),F_{-}(g)\} = \emptyset$ implies that 
the group generated by the support of $\mu_{i}$ 
 contains two independent pseudo-Anosovs $g, hgh^{-1}$.
Hence $\mu_{i}$ is non-elementary.
By continuity of the action of $\mcg$ on $\pmf$, $\nu_{i}$ converges weekly to a $\mu_{\infty}$-stationary measure $\nu_{\infty}$.
Note that as $\mu_{\infty}$ is elementary, we do not have uniqueness of the $\mu_{\infty}$-stationary measure.
But by Lemma \ref{lem.ns2}, we see that $\nu_{\infty} = \delta_{F_{+}(g)}$.
The same argument applied to $\breve\mu_{i}$, we see that $\breve\nu_{\infty} = \delta_{F_{-}(g)}$.
By Lemma \ref{lem.Tseq} applied to $(g^{n}x)_{n\in\mathbb{N}}$ which converges to a uniquely ergodic 
foliation $F_{+}(g)$,
we see that 
for $\breve\nu_{\infty}$-a.e. $\xi$ (in this case $\xi = F_{-}(g)$) and 
$\mathbb{P}_{\infty}$-a.e. $\omega=(\omega_{n})_{n\in\mathbb{N}}$ (in this case $\omega_{n} = g^{n}$),
$\omega=(\omega_{n}b)_{n\in\mathbb{N}}$ and $\xi$ satisfy Condition \ref{cond.seq}.
Hence

$$\ell_{\mathrm{Th}}(\mu_{\infty}) \left(= \ell_{\mathrm{Th}}(g)\right) = \int c^{\mathrm{Th}}_{B}(g,\xi)d\breve\mu_{\infty}(g)d\breve\nu_{\infty}(\xi).$$
As $\breve L_{\mathrm{Th}}(\mu_{i})\rightarrow \breve L_{\mathrm{Th}}(\mu_{\infty})$, Theorem \ref{thm.conti} applies.
\end{proof}
\begin{rmk}
In $\mcg$, there are many elements with fixed point in $\mathcal{T}(S)$.
Let $a\in\mcg$ be such an element with fixed point $b\in\mathcal{T}(S)$.
We can also find a pseudo-Anosov $g$ whose axis does not path through $b$
(in fact axises of ``generic'' pseudo-Anosov elements do not pass through any point in $\mathcal{T}(S)$
 that can be fixed by an element of $\mcg$ \cite{Mas2,Mas3}).
 Then the probability measure $\mu_{i}$ defined by $\mu_{i}(a) = 1-1/i$,  
 and $\mu_{i}(g)=1/i$ is non-elementary, and converges to $\delta_{a}$.
One sees that $L(\mu_{i})\rightarrow L(\delta_{b})=0$ 
as the support of $\mu_{i}$ is the same finite set  for all $i\in\mathbb{N}$.
 Hence by Proposition \ref{prop.zero}, we see that the drift of non-elementary probability measure can be arbitrarily small.
On the other hand by Theorem \ref{thm.ns},  
we see that the drift of non-elementary probability measure can be arbitrarily large.
In particular, the drift of non-elementary probability measure can be any positive real number.
\end{rmk}

\appendix
\section{Continuity of entropy in $\mcg$}
Let $\mu$ be a probability measure on $\mcg$, which induces a probability measure $\mathbb{P}$ on $\mcg^{\mathbb{Z}_{+}}$.
The time one entropy $H(\mu)$ is defined by
$$H(\mu):= \sum_{g\in\mcg}\mu(g)(-\log\mu(g)).$$
The function $H$ is subadditive with the convolution.
Hence if $H(\mu)$ is finite, the following quantity is well-defined:
$$h(\mu) := \lim_{n\rightarrow\infty}H(\mu^{*n})/n,$$
and $h(\mu)$ is called the {\em asymptotic entropy}.
In this appendix, we consider the Teichm\"uller distance and fix a basepoint $b\in\mathcal{T}(S)$.
A probability measure $\mu$ is said to have finite logarithmic moment if 
$\sum_{g \in G} \mu(g)\log d_{\mathcal{T}}(gb,b)$ is finite.
The goal of this section is to prove the following theorem.
\begin{thm}\label{thm.entropy}
Let $\mu_{i}$ be a non-elementary probability measure on $\mcg$ for $ i\in\mathbb{N}\cup\{\infty\} $.
Suppose that $\mu_{i}$ has finite time one entropy and finite logarithmic moment with respect to the Teichm\"uller distance for $i\in\mathbb{N}\cup\{\infty\}$.
Then if a sequence of probability measures $(\mu_{i})_{i\in\mathbb{N}}$ converges simply to $\mu_{\infty}$ and $H(\mu_{i})\rightarrow H(\mu_{\infty})$,
then $h(\mu_{i})\rightarrow h(\mu_{\infty})$.
\end{thm}

%
%

Our strategy is to imitate the argument in \cite{GMM} 
which proves the continuity of asymptotic entropy in the case where the group is Gromov hyperbolic.
We first define {\em shadows} in Teichm\"uller space.
Recall that by the work of Hubbard-Masur \cite{HM}, for a given point $x\in\mathcal{T}(S)$ and $F\in\mathcal{PMF}(S)$, 
there always exists a unique Teichm\"uller geodesic, which we denote by $\Gamma(x,F)$,
emanating from $x$ with corresponding horizontal foliation $F$.
Fixing $b\in\mathcal{T}(S)$ and assigning each Teichm\"uller ray the corresponding horizontal foliation, 
we have a compactification $\overline{\mathcal{T}(S)}$ which is called the Teichm\"uller compactification.

\begin{defi}
Fix $b\in\mathcal{T}(S)$.
Let $C>0$.
Then we define the shadow $\mathcal{O}(y,C)$ of $y\in\mathcal{T}(S)$ with radius $C$ by
$$\mathcal{O}(y,C):=\{F\in\mathcal{PMF}(S)\mid \Gamma(b,F)\cap B_{C}(y)\not= \emptyset\}$$
where $B_{C}(y)$ is the open ball of radius $C$ centered at $y$ with respect to the Teichm\"uller distance.
For $g\in\mcg$, we denote $\mathcal{O}(gb,C)$ by $\mathcal{O}(g,C)$.
Note that as $B_{C}(y)$ is open, $\mathcal{O}(y,C)$ is an open set (one can observe this from \cite[Proposition 11.13]{FaM} combined  with \cite{HM}). 
Hence in particular $\mathcal{O}(y,C)$ is measurable.
\end{defi}
We also let $$\mathbb{S}^{k}:=\{g\in\mcg\mid d_{\mathcal{T}}(b,gb)\in(k-1,k]\}$$ denote the thickened sphere.

We prepare two lemmas.
The first one is a general fact proved in \cite[Lemme 6.5]{Coo}.
Although in \cite{Coo} only hyperbolic groups are discussed, this statement holds in a wider context.
For the convenience of the reader, we give a statement that we need to prove Theorem \ref{thm.entropy} and give a proof.
\begin{lem}[Covering number is finite]\label{lem.cnf}
There exists $D>0$ such that for any $k\in\mathbb{N}$ and $F\in\pmf$,
$$|\{g\in\mathbb{S}^{k}\mid F\in\mathcal{O}(g,C)\}|\leq D.$$
\end{lem}
\begin{proof}
The proof we give here is almost identical to the one in \cite{Coo}.
Let $g_{1}$ and $g_{2}$ be in $\{g\in\mathbb{S}^{k}\mid F\in\mathcal{O}(gb,C)\}.$
Then by the definition of $\mathbb{S}^{k}$ and the shadows,
we see that there are points $p_{i}\in\Gamma(b,F)$ such that $\dt(g_{i}b,p_{i})< C$ for $i=1,2$.
Then by the triangle inequality, we have 
$$\dt(b, g_{i}b)-\dt(g_{i}b, p_{i})
\leq \dt(b,p_{i})\leq
\dt(b, g_{i}b)+\dt(g_{i}b, p_{i})
$$
which implies
$\dt(b,p_{i})\in(k-1-C,k+C]$ for both $i=1,2$.
As $p_{1}$ and $p_{2}$ are on the same geodesic, we have $\dt(p_{1},p_{2})< 2C-1$.
Hence we see that $\dt(g_{1}b,g_{2}b)< 4C-1$.
Now the conclusion follows from the properly discontinuity of the action of $\mcg$ on $\mathcal{T}(S)$.
\end{proof}

We also need to prove that preimage of the shadow has large $\nu$ measure.
\begin{lem}[Preimage of shadows are large]\label{lem.psl}
Let $\mu$ be a non-elementary probability measure on $\mcg$ and $\nu$ the unique $\mu$-stationary measure on $\pmf$.
For any $\epsilon>0$, there exists $C>0$ such that for any $g\in\mcg$
$$\nu(g^{-1}\mathcal{O}(g,C))\geq 1-\epsilon.$$
\end{lem}
\begin{proof}
Suppose the contrary that there exists $\epsilon>0$ such that for any $n\in\mathbb{N}$,
there exists $g_{n}\in\mcg$ such that $\nu(g_{n}^{-1}\mathcal{O}(g_{n},n))<1-\epsilon$.
We let $C_{n}:=\overline{\mathcal{T}(S)}\setminus g_{n}^{-1}\mathcal{O}(g_{n},n)$.
By construction we have, $\nu(C_{n})\geq\epsilon$.
We let $$U:=\bigcap_{k\in\mathbb{N}}\bigcup_{n\geq k}C_{n}.$$
Then $\nu(U)\geq \epsilon>0$, especially $U$ is non-empty.
Let $\xi\in U$ be an arbitrary element.
By the definition of $C_{n}$, we have $\dt(\Gamma(g_{n}b,\xi),b)>n$.
However, by the work of Klarreich \cite[Proposition 5.1]{Kla},
this can only happen if $g_{n}b\rightarrow \xi$ in the Teichm\"uller compactification.
Therefore we have $U=\{\xi\}$.
This contradicts the fact that $\nu$ is non-atomic.
\end{proof}

We use the following theorem, which is stated only for hyperbolic groups in \cite{GMM} but the proof works for more general groups.
\begin{thm}[{\cite{GMM}}]\label{thm.GMM}
Let $G$ be a group of isometries on a metric space $X$, and 
$\mu_{i}$ a probability measure on $G$ with finite logarithmic moment for $i\in\mathbb{N}\cup\{\infty\}$.
Suppose that we have a $\mu_{i}$-boundary $(\partial X,\nu_{i})$ with a unique non-atomic measure $\nu_{i}$ for every $i\in\mathbb{N}\cup\{\infty\}$.
Suppose further that  we can define shadows so that the conclusion of 
Lemma \ref{lem.cnf} and Lemma \ref{lem.psl} holds for large enough $i$.
Then for a sequence of probability measures $(\mu_{i})_{i\in\mathbb{N}}$ converging simply to $\mu_{\infty}$ with $H(\mu_{i})\rightarrow H(\mu_{\infty})$,
we have $h(\mu_{i})\rightarrow h(\mu_{\infty})$.
\end{thm}
\begin{proof}[Sketch of the proof]
This is essentially proved in \cite{GMM}.
We only give comments on how we can get the statement from the discussion in \cite{GMM}.
The condition of the shadows is necessary for the proof of \cite[Theorem 2.6]{GMM}.
Note that we need to replace the word metric $|g|$ with $d(gb,b)$.
Then, by following the same arguments as in \cite[Section 2.4]{GMM}, we get the conclusion.
\end{proof}
\begin{proof}[Proof of Theorem \ref{thm.entropy}]

In \cite{GMM}, the hyperbolicity of the group is used only to have
\begin{enumerate}
\item  the Gromov boundary $\partial G$ equipped with a unique $\mu$-stationary measure $\nu$ is
 a $\mu$-boundary, and
\item the shadows so that the conclusion of Lemma \ref{lem.cnf} and Lemma \ref{lem.psl} holds.
\end{enumerate}
For hyperbolic groups, (1) is proved by Kaimanovich \cite{Kai}, and 
(2) is proved by standard arguments of $\delta$-thin triangles.
For $\mcg$, (1) is given in Theorem \ref{thm.KM}, and
Lemma \ref{lem.cnf} and Lemma \ref{lem.psl} give (2).
Then by Theorem \ref{thm.GMM} we get the conclusion.
\end{proof}

\begin{rmk}
In \cite{GMM}, the following proposition is stated as a work in \cite{Kai}.
\begin{prop}[{\cite[Proposition 2.5]{GMM}}]\label{prop.Kai}
Let $G$ be a hyperbolic group.
Let $\mu$ be a non-elementary probability measure on $G$ with $H (\mu) < \infty$.
Let $\nu$ be its unique stationary measure on $\partial G$.
Define the Martin cocycle on $G\times\partial G$ by $c_{M}(g,\xi) = -log(dg_{*}^{-1}\nu/d\nu)(\xi)$. 
 Then 
 $$h(\mu)\geq
\int_{G\times \partial G}c_{M}(g,\xi)d\mu(g)d\nu(\xi),$$
 with equality only if $\mu$ has finite logarithmic moment.
\end{prop}
In the version 1 on arXiv of \cite{Kai}, 
the Proposition \ref{prop.Kai} is proved in section 10.3 and section 10.4 together with Theorem 16.10.
The discussion in section 10 of \cite{Kai} works for any countable group, 
and the mapping class group version of \cite[Theorem 16.10]{Kai} is \cite[Theorem 2.3.1]{KM}.
\end{rmk}

%


\end{document}